\documentclass[11pt,a4paper,leqno,noamsfonts]{amsart}
\usepackage[english]{babel}
\usepackage[dvipsnames]{xcolor}
\usepackage{graphicx,pifont}
\usepackage[utopia]{mathdesign}
\usepackage{bm}
\usepackage[a4paper,top=3cm,bottom=3cm,
           left=3cm,right=3cm,marginparwidth=60pt]{geometry}
\usepackage[utf8]{inputenc}
\usepackage{braket,caption,comment,mathtools,stmaryrd,enumitem} 
\usepackage[usestackEOL]{stackengine}
\usepackage{multirow,booktabs,microtype,relsize,soul}
\usepackage[foot]{amsaddr}
\usepackage[colorlinks,bookmarks]{hyperref} %
      \hypersetup{colorlinks,%
            citecolor=britishracinggreen,%
            filecolor=black,%
            linkcolor=cobalt,%
            urlcolor=cornellred}
      \numberwithin{equation}{section}
      
\usepackage{thmtools}
\usepackage[capitalise]{cleveref}
\Crefname{conjecture}{Conjecture}{Conjectures}

\input{macros.tex}
\usepackage{tikz}
\usetikzlibrary {patterns,patterns.meta}
\usetikzlibrary{arrows.meta}
\graphicspath{ {./images/} }
\usetikzlibrary{calc}
\definecolor{pastelpurple}{RGB}{180,160,210}
\definecolor{pastelblue}{RGB}{150,180,210}
\definecolor{pastelgreen}{RGB}{180,220,160}
\definecolor{pastelgreen2}{RGB}{140,200,140}
\definecolor{pastelyellow}{RGB}{255,250,160}
\definecolor{pastelorange}{RGB}{255,200,140}
\definecolor{pastelgray}{RGB}{180,200,200}
\definecolor{pastelred}{RGB}{240,150,150}
\definecolor{pastelblack}{RGB}{80,80,80}
\definecolor{darkpastelorange}{RGB}{210,150,110}

\newlength{\figsize}
\title[On the combinatorics of the refined 1-leg DT/PT correspondence]{On the combinatorics of the refined 1-leg DT/PT correspondence}
\author{Davide Accadia, Danilo Lewa\'nski, Sergej Monavari}

\begin{document}
\NewDocumentCommand{\stripedBox}{m O{} O{} O{} m m m}{
    
    \begin{scope}[shift={(#1)}]
        \if\relax\detokenize{#2}\relax
            \draw[dashed] (0,0) rectangle (#5,#6) node[pos=0.5] {#7};
        \else
            \path[draw, thin, fill=#2, 
                path picture={
                    \coordinate (Origin) at (path picture bounding box.south west);
                    \if\relax\detokenize{#3}\relax\else
                        \fill[pattern={Lines[angle=45, 
                              distance={0.2*30pt/sqrt(2)}, 
                              line width=0.2*10.6pt, 
                              yshift=0pt]}, 
                              pattern color=#3] 
                              (path picture bounding box.south west) rectangle (path picture bounding box.north east);
                    \fi
                    \if\relax\detokenize{#4}\relax\else
                        \fill[pattern={Lines[angle=45, distance={0.2*30pt/sqrt(2)}, line width=0.2*7.07pt, yshift=0pt]}, 
                              pattern color=#2] 
                              (path picture bounding box.south west) rectangle (path picture bounding box.north east);
                        \fill[pattern={Lines[angle=45, distance={0.2*30pt/sqrt(2)}, line width=0.2*7.07pt, yshift=-0.2*10pt]}, 
                              pattern color=#3] 
                              (path picture bounding box.south west) rectangle (path picture bounding box.north east);
                        \fill[pattern={Lines[angle=45, distance={0.2*30pt/sqrt(2)}, line width=0.2*7.07pt, yshift=-0.2*20pt]}, 
                              pattern color=#4] 
                              (path picture bounding box.south west) rectangle (path picture bounding box.north east);
                    \fi
                }] (0,0) rectangle (#5,#6) node[pos=0.5] {#7};
        \fi
    \end{scope}
}

\begin{abstract}
We provide a new proof of a result of Bessenrodt on the relation among the generating series of reversed plane partitions and skew plane partitions, motivated by the geometric DT/PT wallcrossing formula for local curves recently proved by the third author. This also recovers a result of Sagan.

We moreover establish various new closed formulas for the weighted enumeration of reversed and skew plane partitions, proving a result dual to a theorem by Gansner, we find a new identity on the generating series counting internal and external hooks of a given Young diagram, and we combine the latter with Bessenrodt's theorem. Finally, we interpret our results as identities in the Fock space via the bosonic/fermionic formalism.
\end{abstract}

\maketitle
{\hypersetup{linkcolor=black}\tableofcontents}

\maketitle

\section{Introduction}
\subsection{Overview}
The study of plane partitions and their various generalisations is a central theme in combinatorics, with deep connections to representation theory and enumerative geometry, in particular for enumerative problems concerning Hilbert schemes, see e.g.~\cite{GMMR2, MOTIVES, Nakajima}. Motivated by the work of the third author on \emph{refined Donaldson-Thomas invariants of local curves} and \emph{topological string partition functions} \cite{IlMona2025}, we offer a new perspective on classical results of Bessenrodt \cite{Bessenrodt} and Gansner \cite{Gansner_reversed_plane_partitions}, which  relate the enumeration problems of reverse plane partitions and skew plane partitions.

Historically, Gansner's work \cite{Gansner_reversed_plane_partitions} refined the algorithm of Hillman and Grassl \cite{Hillman-Grassl} describing the correspondence between hook-lengths and reverse plane partitions.

\subsection{Bessenrodt's result and thin partitions}
Let $\lambda$ be a Young diagram. Denote by $\mathcal{H}(\lambda)$ the set of \emph{internal} hooks of $\lambda$, and by $\mathcal{H}'(\lambda)$ the set of \emph{external} hooks of $\lambda$. We remark that  an internal (resp.~external) hook of $\lambda$ is equivalent to  a $\Box\in \lambda$ (resp.~$\Box\notin \lambda$) where the arm and leg lengths of the hooks (i.e. their hook type) are uniquely determined by the boundary of $\lambda$. Bessenrodt \cite{Bessenrodt} established the following bijection.

\begin{theorem}[{\cite[Thm.~3.2]{Bessenrodt}}]\label{thm: Bessen_original}
    Let $\lambda$ be a Young diagram. There is a bijection of sets
    \begin{align*}
\mathcal{H}'(\lambda) \longleftrightarrow \mathcal{H}'(\varnothing)\cup \mathcal{H}(\lambda),
    \end{align*}
    which preserves hook types.
\end{theorem}

Bessenrodt's original strategy for proving \Cref{thm: Bessen_original} is to carefully keeping track of the relation between addable and removable hooks in $\lambda$ of given hook and arm length. The proof exploits the representation of partitions as Maya diagrams and it admits a natural interpretation in terms of bosonic operators in the Fock space (cf.~also \cite{CB_BijDTPT}).
\smallbreak
Our main result is an alternative proof of \Cref{thm: Bessen_original} for the case of \emph{thin partitions} (cf.~\Cref{thm: Bessen_original_thin}), a  class of partitions which we introduce in \Cref{def:thin}. Our argument is based on a certain subdivision of the Young diagram, in which the constituent rectangles, called tectonic plates, are shifted in a careful and controlled way, that we call tectonic movement. This tectonic movement preserves hook types and makes the underlying combinatorial structure transparent, allowing the desired equality to emerge naturally.

\subsection{Hook--to--strip}
\label{sec:hook:strips:intro}

For $d \geq \ell > 0$, consider the following two collections of pairs of partitions and hooks
    \begin{align*}
         S_{d,\ell} \coloneqq \Big\{ (\lambda, h)  \Big{|} \lambda \in \mathcal{P}_d, h \text{ internal hook of } \lambda, \text{ of hook length } \ell \Big\},
         \\
         S'_{d,\ell} \coloneqq \Big\{ (\lambda, h)  \Big{|} \lambda \in \mathcal{P}_d, h \text{ external hook of } \lambda, \text{ of hook length } \ell \Big\}.
    \end{align*}
These sets are empty whenever $d$ is negative. We prove the following new combinatorial identity relating sets of internal and external hooks.

\begin{theorem}[\Cref{theorem:hooks:strips}]\label{theorem:hooks:strips:intro}
 Let  $d \geq \ell > 0$. There exists a bijection of sets
    \begin{equation*}
        S_{d,\ell} \longleftrightarrow S'_{d-\ell,\ell},
    \end{equation*}
    which preserves the positions of both hands and feet of the hooks, in particular preserving both content sets and hook types.
\end{theorem}

\subsection{Wall--crossing}
The bijection of  \Cref{thm: Bessen_original}, by taking the plethystic exponential of the corresponding generating series, can be equivalently expressed in the following form, see  \Cref{sec: notation partitions} for the definitions and notations.

 \begin{theorem}\label{thm:intro_wallcross} 
 Let $\lambda$ be a Young diagram. We have the identity
     \begin{equation*}
         \prod_{ \Box \in \mathcal{H}'(\lambda)} \frac{1}{1 - x^{a(\Box) + 1} y^{\ell(\Box)}} 
         =
         \prod_{\Box \in \mathcal{H}'(\varnothing)} \frac{1}{1 - x^{a(\Box) + 1}y^{\ell(\Box)}}
         \prod_{\Box \in \mathcal{H}(\lambda)} \frac{1}{1 - x^{a(\Box) + 1}y^{\ell(\Box)}}.
     \end{equation*}
\end{theorem}

Specialising $x=y=q$ one obtains the following statement, originally derived by Sagan \cite[Thm.~2.1]{Sagan}  by means of the so-called \textit{jeu de taquin} algorithm.

\begin{corollary}\label{thm:sagan} Let $\lambda$ be a Young diagram. We have the identity
     \begin{equation*}
         \prod_{ \Box \in \mathcal{H}'(\lambda)} \frac{1}{1 - q^{h(\Box)}} 
         =
         \prod_{\Box \in \mathcal{H}'(\varnothing)} \frac{1}{1 - q^{h(\Box)}}
         \prod_{\Box \in \mathcal{H}(\lambda)} \frac{1}{1 - q^{h(\Box)}}.
     \end{equation*}
\end{corollary}
In other words, Bessenrodt's theorem \cite[Thm.~3.2]{Bessenrodt} specialises to Sagan's theorem \cite[Thm.~2.1]{Sagan} when forgetting the decomposition of the hook length $h(\Box)$ into its arm and leg length, $a(\Box) + \ell(\Box) + 1 = h(\Box)$. Notice that the second product can equivalently be reformulated via the MacMahon generating series. One could also wonder whether  further refinements of \Cref{thm:intro_wallcross}  hold true. In particular, the refinement\footnote{
This is content of Conjecture 2.5 of  the Arxiv v1  version of \cite{IlMona2025}.} in which the monomials in $x$ and $y$ are replaced by the variables $q_{\Box}$ keeping track of the coordinates of each individual box $\Box$ involved in the hooks does not hold true, as one can see e.g. by considering $\lambda=(1)$.

Taking the plethystic exponential of the associated generating series of \Cref{theorem:hooks:strips:intro} yields the following corollary.

\begin{corollary}[\Cref{cor:hook:strip:refined}]
\label{cor:hook:strip:refined:intro}
Let $d \geq \ell > 0$. We have the identity
     \begin{equation*}
         \prod_{(\lambda, \Box) \in S_{d,\ell}}\frac{1}{(1 - x^{a_{\lambda}(\Box) + 1} y^{\ell_{\lambda}(\Box)})}
         =
         \prod_{(\lambda', \Box') \in S'_{d-\ell,\ell}}\frac{1}{(1 - x^{a_{\lambda}(\Box') + 1} y^{\ell_{\lambda}(\Box')})}.
     \end{equation*}
\end{corollary}
Finally, by combining \Cref{thm:intro_wallcross} and \Cref{cor:hook:strip:refined:intro} we obtain we following  result.
\begin{prop}[\Cref{prop: ultimate wall refined}]
\label{prop:hook:strip:refined:intro}
    We have the identity
    \begin{equation*}
        \prod_{\substack{(\lambda, \Box) \in S_{d+\ell,\ell} \\ \ell >0 }}\frac{1}{1 - x^{a_{\lambda}(\Box) + 1} y^{\ell_{\lambda}(\Box)}}
        =
        \left(
        \prod_{\Box'' \in \mathcal{H}'(\varnothing)} \frac{1}{1 - x^{a_{\varnothing}(\Box'') + 1} y^{\ell_{\varnothing}(\Box'')}}
        \right)^{|\mathcal{P}_{d}|}
        \!\!\!\!\!\!
        \prod_{\substack{(\lambda, \Box') \in S_{d,\ell} \\ \ell >0 }}\frac{1}{1 - x^{a_{\lambda}(\Box') + 1} y^{\ell_{\lambda}(\Box')}}.
    \end{equation*}
\end{prop}

\subsection{Relations to the DT / PT correspondence}
The enumeration problem of reverse and skew plane partitions naturally arises in the context of Donaldson-Thomas (DT) and Pandharipande-Thomas (PT) theory of local curves \cite{double-nested-1, Mon_double_nested, IlMona2025} and of surfaces \cite{FGLR}. 
In fact, as the third author showed in \cite{Mon_double_nested, IlMona2025}, the generating series of topological Euler characteristics of Hilbert schemes (\emph{the DT side}) and the moduli spaces of stable pairs (\emph{the PT side}) of a local curve are controlled by the generating series of skew and reverse plane partitions, respectively. From this perspective, the DT/PT \emph{wallcrossing formula} \cite[Thm.~1.9]{IlMona2025} is the geometric counterpart of the combinatorial bijection originally proved by Sagan \cite{Sagan} (see also \cite{Comb_DT/PT}). Under this dictionary, \Cref{thm: Bessen_original} is the \textit{refined} combinatorial counterpart of the \emph{refined} DT/PT wallcrossing formula proved in \cite{IlMona2025}, defined by \emph{virtual integration} on the corresponding moduli space. See also the recent works \cite{long-ref-wall, ref_wall} for the refined DT/PT correspondences for general Calabi-Yau threefolds, and 
\cite{Arb_K-theo_surface, CKM_K, RefTopVertex} for their relation with  the \emph{refined topological vertex}. 


\subsection*{Structure of the paper}
In \Cref{sec:PP} we recall the necessary background on reverse and skew plane partitions: including hook, arm and leg lengths, and their generating functions. \Cref{sec:proof:main} proves the main theorem and constitutes the combinatorial core of the paper. \Cref{sec:non:sottili} contains the statement and the combinatorial proof of the explicit bijection between internal and external hooks of different partitions. Finally, in \Cref{sec:Fock} we briefly discuss the geometric the Fock space operator formalism in which the previous result can be rephrased.

\subsection*{Acknowledgements}
The authors thank R.~Kramer for useful remarks. D.~A. and D.~L. are supported by the University of Trieste and by the INFN within the project MMNLP (APINE). S.~M. is supported by the HORIZON-MSCA-2024-PF-01 Project 101203281 ``Moduli Spaces of Sheaves: Geometry and Invariants'', funded by the Research and Innovation framework programme Horizon Europe. The authors are supported by the INdAM group GNSAGA.

\vspace{1cm}
\section{Reverse and skew plane partitions}
\label{sec:PP}

\subsection{Young diagrams}\label{sec: notation partitions}

A \emph{partition} $\lambda$ of $d\in \BZ_{\geq 0}$ is a finite sequence of positive integers 
\[
\lambda=(\lambda_1\geq \lambda_2\geq \lambda_3\geq \dots),
\]
where 
\[
|\lambda|=\sum_{i}\lambda_i=d.
\]
A partition $\lambda$ can be equivalently pictorially described by its associated \emph{Young diagram}, which is the collection of $d$ \emph{boxes} $\Box$, or \emph{cells}, in $\BZ^2$ located at $(i,j)$ where $0\leq j< \lambda_{i+1}$. We denote by $\mathcal{P}_d$ the collection of Young diagrams of size $d$.

\begin{definition}\label{def: spp}
Let $\lambda$ be a Young diagram.
\begin{itemize}
    \item A \emph{skew plane partition} of shape $\BZ_{\geq 0}^2\setminus \lambda$, or simply of shape $\lambda$, is a collection of  nonnegative integers $\bf{n}=(n_{\Box})_{\Box\in \BZ_{\geq 0}^2\setminus\lambda}$ non-increasing along rows and columns, with only finitely many non-zero entries. In other words, we have  $n_{\Box}\leq n_{\Box'} $ whenever $\Box\geq \Box'$\footnote{This means that if $\Box=(i,j), \Box'=(i',j')$, then $i\geq i'$ and $j\geq j'$.}.
    \item A \emph{reversed plane partition} of shape $\lambda$  is a collection of  nonnegative integers $\mathbf{m}=(m_{\Box})_{\Box\in \lambda}$ non-decreasing along rows and columns. In other words, we have  $m_{\Box}\leq m_{\Box'} $ if $\Box\leq \Box'$.
\end{itemize}
\end{definition}

The \emph{size} of a skew (resp.~reversed) plane partition $\bf{n}$ (resp. $\bf{m}$) is
\begin{align*}
    |\mathbf{n}|=\sum_{\Box\in \BZ^2_{\geq 0}\setminus\lambda}n_{\Box}
     \qquad 
    (\text{resp. }|\mathbf{m}|=\sum_{\Box\in \lambda}m_{\Box}).
\end{align*}
\begin{figure}[!ht]
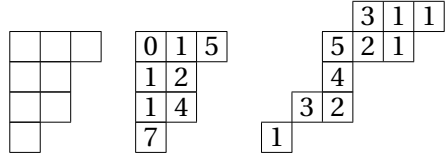

    \centering
   \yng(3,2,2,1) \quad \young(015,12,14,7) \quad \young(:::311,::521,::4,:32,1)     
    \caption{Respectively from the left, a  Young diagram of size 8, a reversed plane partition of size 21 and a skew plane partition of size 23.}
    \label{fig:my_label}
\end{figure}

\subsubsection{Hooks}
 Let $\lambda$ be a Young diagram. We define the   \emph{(internal) hook} at $\Box=(i,j)\in \lambda$ with respect to $\lambda$ to be the collection of boxes  in $\lambda$ below or at the right of $\Box$, namely
\begin{align*}
  H_\lambda(\Box)=\Set{(l,k)\in \lambda| l=i,\, k\geq j \mbox{ or } l\geq i,\, k=j}.
\end{align*}
Similarly, we define the \emph{(external) hook} at  $\Box=(i,j)\in \BZ^2_{\geq 0}\setminus \lambda$ with respect to $\lambda$ to be the collection of boxes in $  \BZ^2_{\geq 0}\setminus \lambda$ which are above or on the left of $\Box$, namely
\begin{align*}
  H_\lambda(\Box)=\Set{(l,k)\in \BZ^2_{\geq 0}\setminus \lambda| l=i,\, k\leq j \mbox{ or } l\leq i,\, k=j}.
\end{align*}

With abuse of notation, we sometimes identify the hook $H_\lambda(\Box)$ simply by $\Box$ whenever it is clear from the context. We denote by $\mathcal{H}(\lambda)$ (resp.~$\mathcal{H}'(\lambda)$ ) the set of internal (resp.~external) hooks of $\lambda$.

Let $\Box=(i,j)\in \lambda$. We define the \emph{arm length} $a(\Box)$ and the   \emph{leg length} $\ell(\Box)$ with respect to $\lambda$ by
\begin{align*}
    a_{\lambda}(\Box)&=|\set{(i,j')\in \lambda| j<j'}|,\\
    \ell_{\lambda}(\Box)&=|\set{(i',j)\in \lambda|i<i'}|.
\end{align*}
Similarly for $\Box \notin \lambda$ we define the \emph{arm length} $a(\Box)$ and the   \emph{leg length} $\ell(\Box)$ with respect to $\lambda$ by
\begin{align*}
    a_{\lambda}(\Box)&=|\set{(i,j')\notin \lambda| j'< j }|,\\
    \ell_{\lambda}(\Box)&=|\set{(i',j)\notin \lambda| i' < i }|.
\end{align*}
We define the \emph{hook length} $h(\Box)$ with respect to $\lambda$ by
\[
h_{\lambda}(\Box) = a_{\lambda}(\Box)+\ell_{\lambda}(\Box)+1,
\]
and the \emph{hook type} of a hook corresponding to $\Box$ as the pair $(a_{\lambda(\Box)}, \ell_{\lambda(\Box)})$. We will often omit the subscript $\lambda$ whenever it is clear from the context. 

Finally, we define the \emph{content} (resp.~\emph{cocontent}) of a cell $\Box = (i,j)$ to be $c(\Box)=j-i$ (resp.~${\text{coc}}(\Box)=j+i$), and the content set of a hook $h$ as the set $\{j-i \,|\, (i,j) \in h \}$. Moreover, the \textit{hand} of a hook is the last cell at the end of its arm, and the \textit{foot} of a hook is the last cell at the end of its leg.
\begin{figure}[!ht]
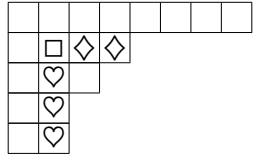

   \centering
   \young(\,\,\,\,\,\,\,\,,\,\Box \diamondsuit\diamondsuit,\,\heartsuit\,,\,\heartsuit,\,\heartsuit)
    \caption{Let $\lambda = (8,4,3,2,2)$. The number of $\diamondsuit$ on the right (resp. of $\heartsuit$ below) of $\Box$ is the arm (resp.~leg) length of the hook defined by $\Box$: $a_{\lambda}(\Box) = 2,\, \ell_{\lambda}(\Box) = 3,\, h_{\lambda}(\Box) = 6$.}
   \label{fig:arm leg}
\end{figure}

\subsection{Generating series}
Let  $\lambda$ be a Young diagram and $\bf{n}$  be  either a  reverse plane partition, or a skew plane partition of shape $\lambda$. Let  $(q_\Box)_{\Box}$  be an infinite set of variables, indexed by $\Box\in \BZ^{2}_{\geq 0} $, and set the multivariable
\[
{\bf{q}}^{\bf{n}}=\prod_{\Box \in \BZ^{2}_{\geq 0} }q_\Box^{n_{\Box}}.
\]
   Let $\dots q_{-1},q_0,q_1, q_2, \dots, $ be an infinite set of variables,  and for a box $\Box\in \BZ^2_{\geq 0}$,  set the multivariable
   \begin{align}\label{eqn:multvar}
\mathbf{p}_{\Box, \lambda}=\prod_{(i,j)\in H_{\lambda}(\Box)}q_{j-i}.       
   \end{align}

\begin{prop}\label{prop: gen series unrefined multiv}
   Let $\lambda$ be a Young diagram, then we have the equalities of generating series:
    \begin{align*}
        \sum_{\bf{n} \in \rpp(\lambda)} {\bfq}^{\bf{n}} \Big{|}_{q_{\Box}=q_{j-i}} &= \prod_{\Box \in \lambda} \frac{1}{1 - \mathbf{p}_{\Box, \lambda}},
             \\
         \sum_{\bf{n} \in \spp(\lambda)} {\bf{q}}^{\bf{n}} \Big{|}_{q_{\Box}=q_{j-i}} &= \prod_{\Box \notin \lambda} \frac{1}{1 - \mathbf{p}_{\Box, \lambda}}.
    \end{align*}
\end{prop}
\begin{proof}
    The first identity was proved by Gansner {\cite[Thm.~5.1]{Gansner_reversed_plane_partitions}}. The second identity follows dually by the first one. In fact,    let $N>0$ be an integer large enough so that $\lambda\subset [0, N]\times [0, N]$. Consider the subcollection $\spp^N(\lambda)\subset \spp(\lambda)$ consisting of skew plane partitions whose boxes with non-zero entries are contained in $[0, N]\times [0, N] $. Consider the Young diagram 
\[
\lambda^D=\set{(N-i, N-j)\in \BZ^2_{\geq 0}| (i,j)\in [0, N]\times [0, N]\setminus \lambda}\subset \BZ^2_{\geq 0}.
\]
Clearly, to each skew plane partition $\bf{n}$ of shape $\lambda$ in $\spp^N(\lambda)$ uniquely corresponds a reverse plane partition $\bf{n}^D$ of shape $\lambda^D$. 

To each box $\Box=(i,j)\notin \lambda$ corresponds a box $\Box^D=(N-i, N-j)\in \lambda^D$. In particular,  we have that 
\[
\mathbf{p}_{\Box, \lambda}=\mathbf{p}_{\Box^D, \lambda^D}.
\]
Therefore, by the first identity we have that 
    \begin{align*}
         \sum_{{\bf{n}} \in \spp^N(\lambda)} {\bf{q}}^{\bf{n}} &= \sum_{{\bf{n}}^D \in \rpp(\lambda^D)} {\bf{q}}^{{\bf{n}}^D} \\
        &=\prod_{\Box^D \in \lambda^D} \frac{1}{1 - \mathbf{p}_{\Box^D, \lambda^D}}\\
         &=\prod_{\Box\in [0,N]\times [0, N] \setminus \lambda} \frac{1}{1 - \mathbf{p}_{\Box, \lambda}}.
    \end{align*}
    Taking the limit for $N\to \infty$ concludes the proof.
\end{proof}
\begin{remark}
    Setting $q_k=q$, the first series  of \Cref{prop: gen series unrefined multiv} admits an arm and leg lengths refined formula
   \begin{align*}
             \sum_{\bf{n} \in \rpp(\lambda)} q^{|\bf{n}|}t^{f(\bf{n})} 
             &=
             \prod_{\Box \in \lambda} \frac{1}{1 - q^{ h(\Box)}t^{\ell(\Box) - a(\Box) - 1}}
             =
             \prod_{\Box \in \lambda} \frac{1}{1 -x^{a(\Box)+1}y^{\ell(\Box)}},
    \end{align*}
    where we applied the change of variables
    \begin{align*}
        \begin{cases}
            x=qt^{-1}
            \\
            y=qt
        \end{cases},
    \end{align*}
    while $f(\bf{n})$ is a suitable statistical weight for reverse plane partitions, originally computed in the context of the \emph{refined topological vertex} \cite{RefTopVertex}, see also \cite[Sec.~4.3.3]{Arb_K-theo_surface}. 

    The statistical weight appearing in the generating series above can, alternatively, be restated in terms of cocontent of hooks hands and feet in the following way
    \begin{align*}
             \sum_{\bf{n} \in \rpp(\lambda)} q^{|\bf{n}|}t^{f(\bf{n})} 
             &= 
             \sum_{\substack{\Box \in \lambda \\ n({\Box}) \geq 0}} q^{\sum_{\Box \in \lambda } n({\Box}) (c(\text{hand($\Box$)}) - c(\text{foot($\Box$)}) + 1)}
             t^{\sum_{\Box \in \lambda } n({\Box}) (\text{coc}(\text{hand($\Box$)}) - \text{coc}(\text{foot($\Box$)}) - 1)},
    \end{align*}
    where the summation sets are in bijection by Gansner correspondence in \Cref{prop: gen series unrefined multiv} of reverse plane partitions as stacks of strips, thus making 
    \begin{equation*}
        f({\bf{n}}) = \sum_{\Box \in \lambda } n({\Box}) (\text{coc}(\text{hand($\Box$)}) - \text{coc}(\text{foot($\Box$)}) - 1)
    \end{equation*} 
    explicit. Using the same ideas as in the proof of the second statement of the same proposition, one can derive an analogous formula for skew plane partitions. The \emph{unrefined limit}, i.e. $t=1$, recovers the formulas in \Cref{prop: gen series unrefined multiv}.
\end{remark}

\vspace{1cm}
\section{Proof of the main theorem}
\label{sec:proof:main}

\subsection{Thin partitions}
Let $\lambda$ be a partition. Consider the subdivision $S_{\lambda}$ of $\mathbb{R}^2_{\geq 0}$ in (possibly semi-infinite) rectangles induced by $\lambda$ by prolonging all horizontal and vertical segment of $\lambda$ to infinite lines, see \Cref{fig:Theta}.

    We call \textit{$\lambda$-tectonic plates} the (possibly semi-infinite) rectangles of this subdivision that lie outside of $\lambda$, and \emph{ $\lambda$-tiles} the (finite) rectangles that lie inside $\lambda$. If we set $\lambda = (1^{m_1}, 2^{m_2}, \dots)$, where $m_i$ denotes the multiplicity of the part $i$,  let $K$ be the number of distinct parts in $\lambda$,  then 
$\BR^2_{\geq 0}$ is divided in exactly $(K+1)^2$ rectangles, $\binom{K+1}{2}$ $\lambda$-tiles and $\binom{K}{2} + 2K + 1$  $\lambda$-tectonic plates, out of which $\binom{K}{2}$ are finite and $2K+1$ are semi-infinite.

    Denote by
    \[
    x_1, \dots, x_{K}, \qquad \qquad y_1, \dots, y_{K},
    \]
    the horizontal (resp.~the vertical) lengths of the $\lambda$-tiles starting from the origin.       
    We define the \emph{arm length} $a_{\lambda}(\bullet)$ (resp.~\emph{leg length} $\ell_{\lambda}(\bullet)$) of a $\lambda$-tile as the sum of the horizontal (resp.~ vertical) lengths $x_i$ (resp.~ $y_j$) of $\lambda$-tiles to the external boundary of $\lambda$. 
    
    Analogously, we define the \emph{arm length} $a'_{\lambda}(\bullet)$ (resp.~\emph{leg length} $\ell'_{\lambda}(\bullet)$) of a $\lambda$-tectonic plate as the sum of the horizontal (resp.~ vertical) lengths $x_i$ (resp.~$y_j$) of $\lambda$-tectonic plates to the external boundary of $\lambda$. Arm and leg lengths of any rectangle in $\mathbb{R}^2_{\geq 0}$ are defined analogously.
    Clearly, $\lambda$-tectonic plates and $\lambda$-tiles are uniquely determined inside $\BR^2_{\geq 0}$ by their horizontal and vertical lengths and by their arm and leg lengths. Finally, we label the $\lambda$-plates in  $R\in S_\lambda$ with coordinates $(i,j)$, as in \Cref{fig:Theta}, for $i,j=1, \dots, K+1$.

    The following definition will be key for the main statement of the paper.
    \begin{definition}\label{def:thin}
    We say that a partition $\lambda$ is \emph{thin} if for all $n=1, \dots, K,$ the following inequalities hold:
    \begin{align*}
          x_1 + \dots+x_n &\leq x_{n+1},\\
          y_1 + \dots+y_n &\leq y_{n+1}.
    \end{align*}
    \end{definition}

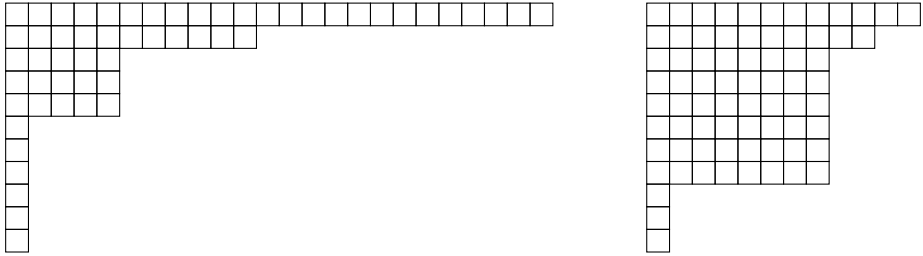
\begin{figure}[h]
    \centering
\begin{tikzpicture}[scale=0.3]

\begin{scope}
\foreach \row/\len in {
0/24,
1/11,
2/5,3/5,4/5,
5/1,6/1,7/1,8/1,9/1,10/1
} {
    \foreach \col in {0,...,\numexpr\len-1} {
        \draw (\col,-\row) rectangle ++(1,1);
    }
}
\end{scope}

\begin{scope}[xshift=800]  
\foreach \row/\len in {
0/12,
1/10,
2/8,3/8,4/8,5/8,6/8,7/8,
8/1,9/1,10/1
} {
    \foreach \col in {0,...,\numexpr\len-1} {
        \draw (\col,-\row) rectangle ++(1,1);
    }
}
\end{scope}

\end{tikzpicture}
    \caption{
    The partition $\lambda = (24,11,5^3,1^6)$ on the left is thin, whereas the partition $\mu = (12,10,8^6,1^3)$ on the right is not.}
    \label{fig:placeholder}
\end{figure}

\subsection{Tectonic movement}
  Let $\lambda$ be partition with Young diagram\footnote{Here we denote by $Y_{\lambda}\subset \BR^2_{\geq 0}$ the "real points" of the Young diagram.} $Y_{\lambda}\subset \BR^2_{\geq 0}$. We define the \emph{tectonic movement} $\Theta_\lambda$ as follows
  \begin{equation}\label{eq:Theta}
    \begin{tikzcd}[row sep =tiny]
   \Theta_{\lambda}: \; \BR^{2}_{\geq 0}\setminus Y_{\lambda}  \arrow[r, " \Theta_\lambda"] &  \BR^2_{\geq 0} \\
    (a,b)\arrow[r,mapsto] &\left(a-\sum_{k=1}^{K+1-j}y_k,b-\sum_{k=1}^{K+1-i}x_k\right),
\end{tikzcd}
\end{equation}
  where $(a,b)\in R$ and $R$ is a $\lambda$-tectonic plate with  coordinates $(i,j)$. 

\begin{remark}
\label{rmk:Theta}
    Pictorially, a $\lambda$-tectonic plate $T$ gets shifted north-west by the tectonic movement $\Theta_{\lambda}$ in the following way: $T$ is shifted north by the height of $\lambda$ north of $T$, and in the same way $T$ is shifted west by the width of $\lambda$ west of $T$. For a graphical example see \Cref{fig:Theta}.

    The geometric motivation behind the definition of the tectonic movement is that $\Theta_\lambda$ is the only map such that a cell representing an external hook to $\lambda$ gets mapped to a cell representing an external hook to the empty partition, of the same hook type.
\end{remark}
  
\begin{figure}[!ht]
    \centering
    \begin{tikzpicture}[scale =0.3]
    \fill[pastelgreen] (0,0) rectangle (15,-6);
    \fill[pastelgreen] (0,0) rectangle (10,-10);
    \fill[pastelgreen] (0,0) rectangle (5,-13);
    \fill[pastelgreen] (0,0) rectangle (3,-16);
    \fill[pastelred] (5,-13) rectangle (10,-16);
    \node[text=black] at (7.5,-14.5) {$R_{3,4}$};
        \node[text=black] at (1.5,-3) {$R_{1,1}$};
        \node[text=black] at (12.5,-3) {$R_{4,1}$};          
        \node[text=black] at (12.5,-3) {$R_{4,1}$};             
        \node[text=black] at (17.5,-18) {$R_{5,5}$};
    \draw[<->] (3,-14.5) -- (0,-14.5);
    \draw[<->] (7.5,0) -- (7.5,-10);
    \draw[-, line width=0.5mm] (0,0) -- (20,0) -- (20,-20) -- (0,-20) -- (0,0)-- (20,0);
    \draw[-,line width = 0.5mm] (0,0) -- (15,0) -- (15,-6) -- (10,-6) -- (10,-10) -- (5,-10) -- (5,-13) -- (3,-13) -- (3,-16) -- (0,-16) -- (0,0) -- (15,0);
    \draw[-,dashed] (0,0) -- (0,-20);
    \draw[-,dashed] (3,0) -- (3,-20);
    \draw[-,dashed] (5,0) -- (5,-20);
    \draw[-,dashed] (10,0) -- (10,-20);
    \draw[-,dashed] (15,0) -- (15,-20);
    \draw[-,dashed] (0,0) -- (20,0);
    \draw[-,dashed] (0,-6) -- (20,-6);
    \draw[-,dashed] (0,-10) -- (20,-10);
    \draw[-,dashed] (0,-13) -- (20,-13);
    \draw[-,dashed] (0,-16) -- (20,-16);
\end{tikzpicture}
\hfill
\begin{tikzpicture}[scale =0.3]
    \fill[pastelgreen] (0,0) rectangle (15,-6);
    \fill[pastelgreen] (0,0) rectangle (10,-10);
    \fill[pastelgreen] (0,0) rectangle (5,-13);
    \fill[pastelgreen] (0,0) rectangle (3,-16);
    \draw[-, line width=0.5mm] (0,0) -- (20,0) -- (20,-20) -- (0,-20) -- (0,0)-- (20,0);
    \draw[-,line width = 0.5mm] (0,0) -- (15,0) -- (15,-6) -- (10,-6) -- (10,-10) -- (5,-10) -- (5,-13) -- (3,-13) -- (3,-16) -- (0,-16) -- (0,0)-- (15,0);
    \draw[-,dashed] (0,0) -- (0,-20);
    \draw[-,dashed] (3,0) -- (3,-20);
    \draw[-,dashed] (5,0) -- (5,-20);
    \draw[-,dashed] (10,0) -- (10,-20);
    \draw[-,dashed] (15,0) -- (15,-20);
    \draw[-,dashed] (0,0) -- (20,0);
    \draw[-,dashed] (0,-6) -- (20,-6);
    \draw[-,dashed] (0,-10) -- (20,-10);
    \draw[-,dashed] (0,-13) -- (20,-13);
    \draw[-,dashed] (0,-16) -- (20,-16);

    \fill[pastelred] (2,-3) rectangle (7,-6);
    \node[text=black] at (4.5,-4.5) {$\Theta(R_{3,4})$};

\end{tikzpicture}
    \caption{Example of tectonic movement.}
    \label{fig:Theta}
\end{figure}
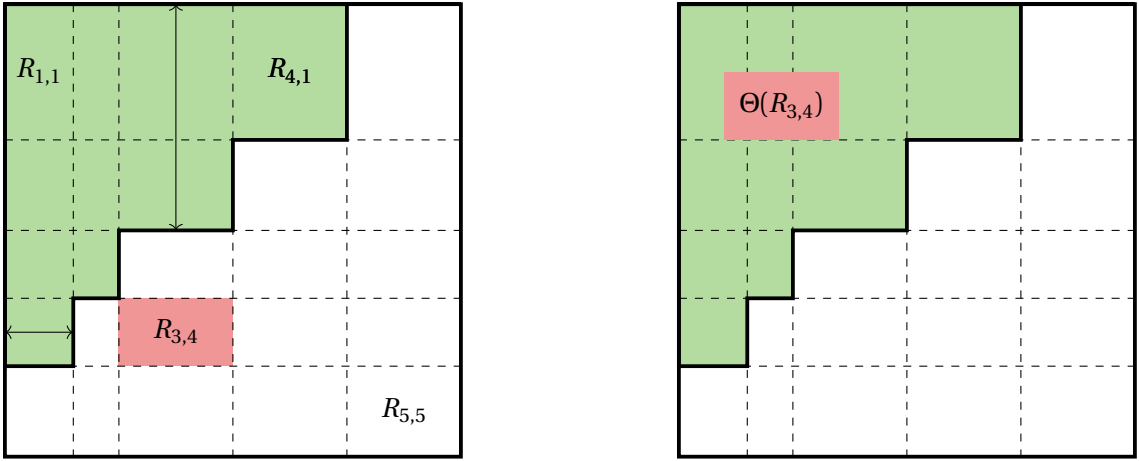

We prove two preliminary results on the tectonic movement. The first one describes the intersections of the $\lambda$-tectonic plates after the tectonic movement whenever $\lambda$ is a thin partition. 

   \begin{prop} 
   \label{prop:123:stella}
    Let $\lambda$ be a thin partition. 
        \begin{itemize}
            \item[i).] Let $R_{i,j}, R_{\alpha, \beta}$ be $\lambda$-tectonic plates. The intersection 
            \begin{equation*}
                \Theta_\lambda(R_{i,j}) \cap \Theta_\lambda(R_{\alpha,\beta})
            \end{equation*} 
            has measure zero except in the following cases:
            \begin{enumerate}
                \item $\alpha=i+1$ and $\beta = j-1$ (or $\alpha=i-1$ and $\beta = j+1$);
                \item $i+j = \alpha+\beta = K+2$.
            \end{enumerate}
            \item[ii).] Let $R_{K+2-j,K+1-i}$ and $R_{K+1-j,K+2-i}$ be 
                   $\lambda$-tectonic plates sharing a vertex. Then the  intersection 
            \begin{equation*}
                \Theta_\lambda(R_{K+2-j,K+1-i}) \cap \Theta_\lambda(R_{K+1-j,K+2-i})
            \end{equation*} 
            is a rectangle of the same size of $R_{i,j}$ and whose cells are in bijection to the internal hooks of $\lambda$ identified by cells of $R_{i,j}$.
            \item[iii).]Let $R_{i,j}, R_{\alpha, \beta}$ be $\lambda$-tectonic plates,with $i+j = \alpha+\beta = K+2$. Then the intersection 
            \begin{equation*}
            \Theta_\lambda(R_{i,j}) \cap \Theta_\lambda(R_{\alpha,\beta})
            \end{equation*} 
            is a rectangle of positive measure containing the origin.
         \end{itemize}
    \end{prop}
    \begin{proof} 
    Given two $\lambda$-tectonic plates, the intersection of their $\Theta$-image is a (possibly degenerate or empty) rectangle $A$ of height $y(A)$ and the width of $x(A)$ given by:
    \begin{align*}
        x(A) &= 
        \begin{cases}
            x_i, \qquad \text{ if } \sum_{t = K+2 -j}^{i-1} x_t = \min\left(\sum_{t = K+2 -j}^{i-1} x_t, \sum_{t = K+2 -\beta}^{\alpha-1} x_t\right)
            \\
            x_{\alpha}, \qquad \text{ if } \sum_{t = K+2 -\beta}^{\alpha-1} x_t = \min\left(\sum_{t = K+2 -j}^{i-1} x_t, \sum_{t = K+2 -\beta}^{\alpha-1} x_t\right)
        \end{cases}
        \!\!\!\! - \;\;
        \Big{|}
        \sum_{t = K+2 -j}^{i-1} \!\!\!\!\!\!\!\! x_t - \!\!\!\!\!\!\!\! \sum_{t = K+2 -\beta}^{\alpha-1} \!\!\!\!\!\!\!\! x_t
        \Big{|}
        \\
        y(A) &= 
        \begin{cases}
            y_j, \qquad \text{ if } \sum_{t = K+2-i}^{j-1} y_t = \min \left( \sum_{t = K+2-i}^{j-1} y_t, \sum_{t = K+2-\alpha}^{\beta-1} y_t \right)
            \\
            y_{\beta}, \qquad \text{ if } \sum_{t = K+2-\alpha}^{\beta-1} y_t = \min\left( \sum_{t = K+2-i}^{j-1} y_t, \sum_{t = K+2-\alpha}^{\beta-1} y_t \right).
        \end{cases}
        \!\!\!\! - \;\;
        \Big{|}
        \sum_{t = K+2-i}^{j-1} \!\!\!\!\!\!\!\! y_t - \!\!\!\!\!\!\!\! \sum_{t = K+2-\alpha}^{\beta-1} \!\!\!\!\!\!\!\! y_t .
        \Big{|}
    \end{align*}
    If at least one between $x(A)$ and $y(A)$ is zero (resp. negative), then $A$ is of measure zero (resp. empty). Notice that, by the thinness condition, each minimum can be resolved by selecting the sum with the highest index value. Therefore, the formulae above simplify to:
    \begin{align*}
    x(A) &= 
    \begin{cases}
    \sum_{t = K+2 -j}^{i} x_t - \sum_{t = K+2 -\beta}^{\alpha-1} x_t \quad \text{ if } i < \alpha
    \\
    \sum_{t = K+2 -\beta}^{\alpha} x_t - \sum_{t = K+2 -j}^{i-1} x_t \quad \text{ if } i > \alpha
    \end{cases}
    \\
    y(A) &= 
    \begin{cases}
    \sum_{t = K+2-i}^{j} y_t - \sum_{t = K+2-\alpha}^{\beta-1} y_t
    \quad \text{ if } j < \beta
        \\
    \sum_{t = K+2-\alpha}^{\beta} y_t  - \sum_{t = K+2-i}^{j-1} y_t
    \quad \text{ if } j > \beta
    \end{cases} 
\end{align*}
    Without loss of generality, we can work under the assumption that $i < \alpha$. Let us analyse the first quantity 
    \begin{equation*}
        \sum_{t = K+2 -j}^{i} x_t - \sum_{t = K+2 -\beta}^{\alpha-1} x_t.
    \end{equation*}
    Notice that, since $i+j \geq K+2$, the first sum can never vanish. The second sum, instead, since again $\alpha + \beta \geq K+2$, vanishes if and only if $\alpha+\beta = K+2$; in other words whenever $R_{\alpha,\beta}$ lies on the first antidiagonal. Similarly for the second case of $x(A)$.

    Let us assume that $\alpha + \beta = K+2$. Then $x(A) > 0$. Since $i+j\geq K+2$ and $i < \alpha$ by assumption, we must have $j > \beta$, hence 
    \begin{equation*}
        y(A) = \sum_{t = K+2-\alpha}^{\beta} y_t  - \sum_{t = K+2-i}^{j-1} y_t.
    \end{equation*}
    If $i+j = K+2$ then the second sum vanishes, hence $y(A) > 0$. If instead $i+j > K+2$, then the second sum includes the summand $y_{j-1}$ which is, again by thinness assumption, greater than the sum of any subset of other $y_t$ because all the other indices $t$ appearing are smaller than $j-1$. Therefore $y(A) < 0$ unless $y_{j-1}$ is cancelled out by $y_\beta$, which happens only if $\beta = j-1$. However, if $\beta = j-1$ we have 
    \begin{equation*}
        y(A) = \sum_{t = \beta}^{\beta} y_t  - \sum_{t = K+2-i}^{\beta} y_t = - \sum_{t = \alpha + \beta - i}^{\beta - 1} y_t \leq 0.
    \end{equation*}
    We just proved that if $R_{\alpha,\beta}$ is on the first antidiagonal, then $A$ is of of zero measure unless $R_{i,j}$ is also on the first antidiagonal, and of positive measure in that case. This is the statement of \emph{i).} $(2)$ and also of \emph{iii).}

    Let us now assume that $i+j > K+2$ and $\alpha + \beta > K+2$. We still assume $i < \alpha$ without loss of generality, hence 
    \begin{equation*}
        x(A) = \sum_{t = K+2 -j}^{i} x_t - \sum_{t = K+2 -\beta}^{\alpha-1} x_t.
    \end{equation*}
    Then $x(A) \leq 0$ unless $\alpha - 1 = i$ and $j < \beta$. By looking at the $j<\beta$ case for $y(A)$ one finds that for the same reason $y(A) \leq 0$ unless $j -1 = \beta$. Therefore $A$ is of positive measure only if the $\lambda$-tectonic plates indeed share a vertex and lie on the same antidiagonal, in which case $x(A) = x_{K+3-i} > 0 $ and $y(A) = y_{K+2-j} > 0$. This proves \emph{i).} $(1)$ as well as the statement about the size of the intersection in \emph{ii).} The statement about the conservation of the hook types follows from \Cref{rmk:Theta}. This concludes the proof of the proposition.
    \end{proof}

\begin{prop}
\label{prop:surj}
    Let $\lambda$ be a thin partition. Then  $\Theta_\lambda$ is surjective.
\end{prop}
\begin{proof}
    The result follows from the application of the following two arguments: 
    \begin{itemize}
        \item[1.]\label{eqn:1} if two $\lambda$-tectonic plates share a vertex then the intersection of their $\Theta$-image has non zero measure, by point \emph{i).} of \Cref{prop:123:stella};

        \item[2.] If two $\lambda$-tectonic plates share a vertical edge then their $\Theta$-images share again a portion of a vertical edge. In particular $\Theta_\lambda(R_{(i,j)})\cap \Theta_\lambda(R_{(i,j+1)}) \neq \varnothing \neq \Theta_\lambda(R_{(i,j)})\cap \Theta_\lambda(R_{(i+1,j)})$. In fact, we have that 
         \begin{equation*}
            \ell'_{\varnothing}(\Theta_{\lambda}(R_{(i,j)})) - \ell'_{\varnothing}(\Theta_{\lambda}(R_{(i,j+1)})) = y_{K+1-j}.
    \end{equation*}
    This is true because $i > K+1-j$ since $R_{(i,j)}$ is a $\lambda$-tectonic plate and hence $y_{K+1-j} \leq y_i$ since the partition is thin. Similarly for horizontally adjacent $\lambda$-tectonic plates.
    \end{itemize} 
    
     We now apply these principles in the following way. For $k = 1, \dots, K$, consider the  sets
    \begin{align*}
    AD'_k(\lambda) &\coloneqq \{ R \, \text{ a } \, \lambda\text{-tectonic plate of coordinates $(i,j)$} \, | \,  i + j = K + 1 + k \},
    \end{align*}
    which satisfies
    \[
       \BR^{2}_{\geq 0}\setminus Y_{\lambda} = \bigcup_{k=1}^K AD'_k(\lambda).
    \]
    By item 1. above we have that
    \begin{equation*}
        \Theta_\lambda(AD'_k(\lambda)) = \bigcup_{R \in AD'_k(\lambda)} \Theta(R)
    \end{equation*}
    is simply connected.
    By item $2.$ above we have that
    \begin{equation*}
        \Theta_\lambda(AD'_k(\lambda)) \cup \Theta_\lambda(AD'_{k+1}(\lambda))
    \end{equation*}
    is again simply connected \textemdash \; and hence 
    \begin{equation*}
        \bigcup_{k=1}^K \Theta_\lambda(AD'_k(\lambda)) = \Theta_\lambda\left(\bigcup_{k=1}^K AD'_k(\lambda) \right)
    \end{equation*}
    is simply connected. To conclude the proof, we   notice that 
    \begin{itemize}
        \item the $\lambda$-tectonic plate with coordinates $(K+1,K+1)$ is fixed by $\Theta_\lambda$ and is infinite both in horizontal and vertical length,
        \item the  $\lambda$-tectonic plates  with coordinates $(K+1,j)$ and $(j,K+1)$ for $j=1, \dots, K$ are infinite in either the horizontal or vertical direction.
    \end{itemize}
    These fact imply that if $(a,b)\in \Theta_\lambda\left(\bigcup_{k=1}^K AD'_k(\lambda) \right)$ and\footnote{By this we mean  that $a'\geq a$ and $ b'\geq b$.} $(a',b')\geq (a,b)$, then 
    $$
    (a',b')\in \Theta_\lambda\left(\bigcup_{k=1}^K AD'_k(\lambda) \right)
    $$
    as well. The proof is concluded by noticing that $(0,0) \in \Theta(AD'_1(\lambda))$.
\end{proof}

We are ready to prove Bessenrodt's original result \Cref{thm: Bessen_original} in the case of thin partitions.
\begin{theorem}\label{thm: Bessen_original_thin}
    Let $\lambda$ be a thin partition. There is a bijection of sets
    \begin{align}
    \label{eq:DT:PT:sets}
\mathcal{H}'(\lambda) \longleftrightarrow \mathcal{H}'(\varnothing)\cup \mathcal{H}(\lambda),
    \end{align}
    which preserves hook types.
\end{theorem}
\begin{proof} 
    We set $\Theta=\Theta_\lambda$ to ease the notation. Consider the union of all $\lambda$-tectonic plates $W_\lambda$. The collection of cells inside $W_\lambda$ correspond to the set $\mathcal{H}'(\lambda)$ of external hooks of $\lambda$. We now build the required bijection by moving these cells, packaged in $\lambda$-tectonic plates, via the tectonic movement $\Theta$. More precisely, we will assign the cells in $W_\lambda$   to  $\mathcal{H}'(\varnothing)$ and  $\mathcal{H}(\lambda)$. 

    By \Cref{rmk:Theta} a cell representing an external hook to $\lambda$ is mapped by $\Theta$ to a cell representing an external hook to the empty partition of the same hook type.
    
    By \Cref{prop:surj} we have that $\Theta_{\lambda}$ is surjective on $\mathbb{R}^2_{\geq 0}$, hence $\mathcal{H}'(\varnothing)$ is in bijection with a subset of $\Theta(\mathcal{H}'(\lambda))$, which we denote again by $\mathcal{H}'(\varnothing)$ by slight abuse of notation. Therefore we need to show that $\Theta(\mathcal{H}'(\lambda)) \setminus \mathcal{H}'(\varnothing)$ is in bijection with $\mathcal{H}(\lambda)$. We are going to establish this equality by studying the intersections of the $\Theta$-images of tectonic plates.
    

    The statement is ensured by \Cref{prop:123:stella}: if two plates do not intersect to begin with, then their $\Theta$-images cannot intersect, except if they lie in the first external antidiagonal (point \emph{i).}). If they do, we can still consider the chain of plates in the first antidiagonal and consider the whole intersection of their $\Theta$-images as the union of intersections of $\Theta$-images of pairs of plates sharing a vertex (point \emph{iii).}). If they do not lie in the first antidiagonal, then they only intersect if they lie in the same antidiagonal and share a vertex. In any case, we can reduce the analysis of the entire intersection as the union of intersections of $\Theta$-images of pairs of plates on the same antidiagonal sharing a vertex:
    \begin{equation*}
    \Theta(\mathcal{H}'(\lambda)) \setminus \mathcal{H}'(\varnothing) 
    =
    \bigcup_{\substack{i=1, \dots, K \\ j=2, \dots, K+1 \\ i+j \geq K+2}} \!\!\!\!\! \Theta(R_{i,j}) \; \cap \; \Theta(R_{i+1, j-1})
    \end{equation*}
    
    Now, by point \emph{ii).}, the external hooks of $\Theta(R_{K+2-j,K+1-i}) \; \cap \; \Theta(R_{K+1-j, K+2-i})$ have the same hook type as the internal hooks of $R_{i,j}$. Moreover, there is an obvious bijection between $\lambda$-tiles and such pairs of $\lambda$-plates, constructed by identifying the coordinates 
    \begin{equation*}
        (i,j) \longleftrightarrow ((K+2-j,K+1-i),(K+1-j,K+2-i)).
    \end{equation*}
    This concludes the proof of the theorem.
\end{proof}

\subsection{The $K \leq 2$ case}
We show in this section how to generalise the proof of \Cref{thm: Bessen_original_thin} for a general partition (not necessarily thin) $\lambda$ with $K \leq 2$. 

Although the partitions for  the $K=0,1$ cases are necessarily thin -- and hence they are covered by \Cref{thm: Bessen_original_thin} -- we quickly go through them for completeness. For $K=0$ the partition $\lambda$ must be empty and the result trivially holds.
 For $K=1$ the Young diagram of $\lambda$ is a rectangle. Then, the only two $\lambda$-tectonic plates moving must slide to the origin, intersecting precisely on $\lambda$. Out of the two copies of the rectangle $\lambda$, the hooks of one copy correspond to the external hooks of the empty set, which is the first set in the RHS of \Cref{eq:DT:PT:sets}, the second copy of the rectangle can be rotated by $\pi$ so that its hooks now match the internal hooks of $\lambda$, which is the second set in the RHS of \Cref{eq:DT:PT:sets}.

For $K=2$, there are three cases to be considered: 
\begin{itemize}
    \item[\textit{Case I}] The thin case, where the thinness condition is satisfied in both directions: $$x_1 < x_2 \quad \wedge \quad y_1 < y_2;$$
    \item[\textit{Case II}] The case in which the thinness condition is violated in both directions: $$x_1 > x_2 \quad \wedge \quad y_1 > y_2;$$
    \item[\textit{Case III}] The case in which the thinness condition is satisfied only in one of the two directions: without loss of generality we can assume:
    $$x_1 > x_2 \quad \wedge \quad y_1 < y_2.$$
\end{itemize}

The first case is covered by \Cref{thm: Bessen_original_thin}, therefore we are left with proving the second and third case. The proof consists in keeping track of the tectonic movement of each plate and their intersections, rearranging them in such a way to cover entirely $\BR^2_{\geq 0}$ together with an extra copy of the original Young diagram $\lambda$ containing exactly its internal hooks content.

\subsubsection{Case II} Let $\lambda$ be a partition such that  $x_1 > x_2$ and $y_1 > y_2$. We name the $\lambda$-tiles by antidiagonals according to the first drawing in \Cref{fig:thick}. Let us consider pairs of $\lambda$-tectonic plates and compute the intersections of their image via $\Theta$. 

Notice for instance that the semi-infinite rectangle $P_1$ slides northwards all the way to the $x$-axis, so that $\Theta(P_1)$ covers all of $A, B_1,\text{ and } P_1$. In the same way $P_3$ slides westwards. Therefore the semi-infinite rectangles $\Theta(P_1)$ and $\Theta(P_3)$ overlap covering $A$ twice. With the same logic, $\Theta(P_2)$ is a rectangle of the same size of $P_2$, contained in $A$ and adjacient to both axes.
$\Theta(Q_1)$ intersects with $\Theta(P_3)$, inside of $B_2$, covering it completely except for a $y_1$ gap northwards, due to the presence of the $P_2$ plate. 

Similarly $\Theta(Q_1)$ intersects $\Theta(Q_2)$ exactly in $P_2$, and $\Theta(P_1)$ and $\Theta(Q_2)$ intersect inside $B_1$ leaving a gap westwards of the same width as the width of $P_2$. In the second drawing of \Cref{fig:thick}, the $\Theta$-images of the $\lambda$-tectonic plates are depicted: each colour in each rectangle represent the image of a tectonic plate of the same colour. For instance, the three colours in the rectangle labelled $\Theta(P_2)$ represent the overlapping of the $\Theta$-images of $P_1, P_2$ and $P_3$.

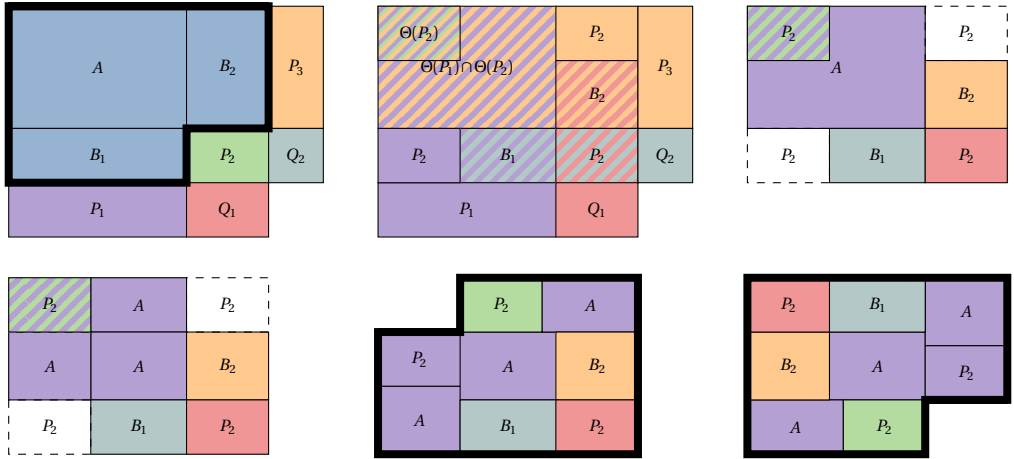
\begin{figure}[h!]
    \centering
    \begin{tikzpicture}[scale=0.18,every node/.style={scale=0.6}]
    \begin{scope}
       
    \stripedBox{0,0}[pastelblue]{13}{-9}{$A$}
    
    \stripedBox{13,0}[pastelblue]{6}{-9}{$B_2$}
    
    \stripedBox{0,-9}[pastelblue]{13}{-4}{$B_1$}
    
    \stripedBox{0,-13}[pastelpurple]{13}{-4}{$P_1$}
    
    \stripedBox{13,-9}[pastelgreen]{6}{-4}{$P_2$}
    \stripedBox{19,0}[pastelorange]{4}{-9}{$P_3$}
    \stripedBox{13,-13}[pastelred]{6}{-4}{$Q_1$}
    \stripedBox{19,-9}[pastelgray]{4}{-4}{$Q_2$}
    \draw[-, line width= 3.0pt] (0,0) -- (19,0) -- (19,-9) --(13,-9) -- (13,-13) -- (0,-13) --(0,0)-- (19,0);

   \end{scope} 
    \begin{scope}[shift={(27,0)}]
  
    \stripedBox{0,0}[pastelorange][pastelpurple]{13}{-9}{$\Theta(P_1) \cap \Theta(P_2)$}
    
    \stripedBox{13,-4}[pastelorange][pastelred]{6}{-5}{$B_2$}
    
    \stripedBox{6,-9}[pastelpurple][pastelgray]{7}{-4}{$B_1$}
    
    \stripedBox{0,-13}[pastelpurple]{13}{-4}{$P_1$}
    
    \stripedBox{0,0}[pastelgreen][pastelorange][pastelpurple]{6}{-4}{$\Theta(P_2)$}
    
    \stripedBox{13,0}[pastelorange]{6}{-4}{$P_2$}
    
    \stripedBox{0,-9}[pastelpurple]{6}{-4}{$P_2$}
    \stripedBox{13,-9}[pastelred][pastelgray]{6}{-4}{$P_2$}
    \stripedBox{19,0}[pastelorange]{4}{-9}{$P_3$}
    \stripedBox{13,-13}[pastelred]{6}{-4}{$Q_1$}
    \stripedBox{19,-9}[pastelgray]{4}{-4}{$Q_2$}

  \end{scope}  
    \begin{scope}[shift={(54,0)}]

    \stripedBox{0,0}[pastelpurple]{13}{-9}{$A$}
    
    \stripedBox{13,-4}[pastelorange]{6}{-5}{$B_2$}
    
    \stripedBox{6,-9}[pastelgray]{7}{-4}{$B_1$}
    
    \stripedBox{0,0}[pastelgreen][pastelpurple]{6}{-4}{$P_2$}
    
    \stripedBox{13,0}{6}{-4}{$P_2$}
    
    \stripedBox{0,-9}{6}{-4}{$P_2$}
    \stripedBox{13,-9}[pastelred]{6}{-4}{$P_2$}

  \end{scope}
    \begin{scope}[shift={(0,-20)}]

    \stripedBox{6,0}[pastelpurple]{7}{-4}{$A$}
    \stripedBox{6,-4}[pastelpurple]{7}{-5}{$A$}
    \stripedBox{0,-4}[pastelpurple]{6}{-5}{$A$}
    
    \stripedBox{13,-4}[pastelorange]{6}{-5}{$B_2$}
    
    \stripedBox{6,-9}[pastelgray]{7}{-4}{$B_1$}
    
    \stripedBox{0,0}[pastelgreen][pastelpurple]{6}{-4}{$P_2$}
    
    \stripedBox{13,0}{6}{-4}{$P_2$}
    
    \stripedBox{0,-9}{6}{-4}{$P_2$}
    \stripedBox{13,-9}[pastelred]{6}{-4}{$P_2$}

  \end{scope}
    \begin{scope}[shift={(27,-20)}]

    \stripedBox{6,-4}[pastelpurple]{7}{-5}{$A$}

    \stripedBox{6,0}[pastelgreen]{6}{-4}{$P_2$}
    \stripedBox{12,0}[pastelpurple]{7}{-4}{$A$}
    
    \stripedBox{0,-4}[pastelpurple]{6}{-4}{$P_2$}
    \stripedBox{0,-8}[pastelpurple]{6}{-5}{$A$}
    
    \stripedBox{13,-4}[pastelorange]{6}{-5}{$B_2$}
    
    \stripedBox{6,-9}[pastelgray]{7}{-4}{$B_1$}
    \stripedBox{13,-9}[pastelred]{6}{-4}{$P_2$}
    \draw[-, line width= 3 pt] (0,-4) -- (6,-4) -- (6,0) -- (19,0) -- (19,-13) -- (0,-13) -- (0,-4)-- (6,-4);

  \end{scope}

    \begin{scope}[shift={(73,-33)},rotate=180]
    \stripedBox{6,-4}[pastelpurple]{7}{-5}{$A$}

    \stripedBox{6,0}[pastelgreen]{6}{-4}{$P_2$}
    \stripedBox{12,0}[pastelpurple]{7}{-4}{$A$}
    
    \stripedBox{0,-4}[pastelpurple]{6}{-4}{$P_2$}
    \stripedBox{0,-8}[pastelpurple]{6}{-5}{$A$}
    
    \stripedBox{13,-4}[pastelorange]{6}{-5}{$B_2$}
    
    \stripedBox{6,-9}[pastelgray]{7}{-4}{$B_1$}
    \stripedBox{13,-9}[pastelred]{6}{-4}{$P_2$}
    \draw[-, line width= 3 pt] (0,-4) -- (6,-4) -- (6,0) -- (19,0) -- (19,-13) -- (0,-13) -- (0,-4)-- (6,-4);

  \end{scope}
\end{tikzpicture}
    \caption{Procedure for a thick partition.}
    \label{fig:thick}
\end{figure}

Notice that, by the definition of tectonic movement $\Theta$, all the cells in the second drawing correspond to external hooks of the empty partition. Therefore, collecting a copy of $\mathbb{R}^2_{\geq 0}$, i.e. collecting any one colour from each rectangle, forms the first set $\mathcal{H}'({\varnothing})$ in the RHS of \Cref{eq:DT:PT:sets}. We then remove a colour from each rectangle, obtaining the third drawing in \Cref{fig:thick}, and we need to show that what is left correponds exactly to $\mathcal{H}({\lambda})$. We are going to show this by cutting and rearranging the remaining rectangles.

Let us cut $A$ in four rectangles as in drawing 4 of \Cref{fig:thick}, and let us slide them as in the drawing 5 of \Cref{fig:thick}. We have now obtained the same shape of the diagram $\lambda$, rotated by $\pi$. Let us observe for instance what happens to the hooks in the green rectangle $P_2$ before and after sliding eastwards (drawing 4 and 5): before sliding in drawing 4 each cell of the green $P_2$ correspond to an external hook of the empty partition, having leg northwards and arm westwards; after sliding in drawing 5 the hooks remain with the same orientation and lengths, so by rotating the whole shape by $\pi$ in drawing 6, the external hooks of the empty partition become internal hooks of $\lambda$. It is easy to see that this holds for each other rectangle, bringing external hooks of the empty partition exactly to the set $\mathcal{H}(\lambda)$ of internal hook of $\lambda$.

\subsubsection{Case III} 
Without loss of generality we can assume that $\lambda$ is a partition such that $x_1>x_2$ and $y_2 >y_1$. We name the tiles by antidiagonals according to the first drawing of \Cref{fig:semithick}.

The proof goes exactly as in \emph{Case II}: the second drawing represents the intersection of the tectonic movement of the tectonic plates, from the second to the third drawing we removed a colour per rectangle to account for $\mathcal{H}'(\varnothing)$, in the fourth drawing we cut the rectangle $B$ in two parts and we rearrange the rectangles in drawing 5, obtaining the rotated shape of the Young diagram $\lambda$. In drawing 6 we rotate the shape by $\pi$, matching the cells with internal hooks of $\lambda$, hence obtaining $\mathcal{H}(\lambda).$ This concludes the proof for $K \leq 2$ for not necessarily thin partitions $\lambda$.

\begin{figure}[h!]
    \centering
    \begin{tikzpicture}[xscale=0.18, yscale=0.3,every node/.style={scale=0.6}]
    \begin{scope}[shift={(0,0)}]
    
    \stripedBox{0,0}[pastelblue]{10}{-2}{$A$}
    
    \stripedBox{10,0}[pastelblue]{4}{-2}{$B_2$}
    
    \stripedBox{0,-2}[pastelblue]{10}{-3}{$B_1$}
    
    \stripedBox{0,-5}[pastelpurple]{10}{-4}{$P_1$}
    \stripedBox{10,-2}[pastelgreen]{4}{-3}{$P_2$}
    \stripedBox{14,0}[pastelorange]{4}{-2}{$P_3$}
    \stripedBox{10,-5}[pastelred]{4}{-4}{$Q_1$}
    \stripedBox{14,-2}[pastelgray]{4}{-3}{$Q_2$}
    \draw[-, line width= 3.0pt] (0,0) -- (14,0) -- (14,-2) --(10,-2) -- (10,-5) -- (0,-5) --(0,0) -- (14,0);
\end{scope}
    \begin{scope}[shift={(25,0)}]
    \stripedBox{0,0}[pastelorange][pastelpurple]{10}{-2}{$A$}
    \stripedBox{10,0}[pastelorange]{4}{-2}{$B_2$}
    \stripedBox{10,-2}[pastelgray]{4}{-1}{$B_2$}
    
    \stripedBox{0,-2}[pastelpurple]{10}{-3}{$B_1$}
    \stripedBox{4,-2}[pastelpurple][pastelgray]{6}{-3}{$B_1$}
    
    \stripedBox{10,-3}[pastelgray][pastelred]{4}{-2}{$B_2$}
    \stripedBox{0,-5}[pastelpurple]{10}{-4}{$P_1$}
    \stripedBox{0,0}[pastelgreen][pastelorange][pastelpurple]{4}{-3}{$P_2$}
    \stripedBox{0,-2}[pastelgreen][pastelpurple]{4}{-1}{}
    \stripedBox{14,0}[pastelorange]{4}{-2}{$P_3$}
    \stripedBox{10,-5}[pastelred]{4}{-4}{$Q_1$}
    \stripedBox{14,-2}[pastelgray]{4}{-3}{$Q_2$}
\end{scope}
    \begin{scope}[shift={(50,0)}]
    \stripedBox{0,0}[pastelorange]{10}{-2}{$A$}
    \stripedBox{4,-2}[pastelgray]{6}{-3}{$B_1$}
    \stripedBox{10,0}{4}{-3}{}    
    \stripedBox{10,-3}[pastelred]{4}{-2}{$B_2$}
    \stripedBox{0,0}[pastelgreen][pastelorange]{4}{-3}{$P_2$}
    \stripedBox{0,-2}[pastelgreen]{4}{-1}{}
    \stripedBox{0,-3}{4}{-2}{}    
\end{scope}
    \begin{scope}[shift={(0,-13)}]
    \stripedBox{0,0}[pastelorange]{10}{-2}{$A$}
    \stripedBox{4,-2}[pastelgray]{6}{-1}{$B_1$}
    \stripedBox{4,-3}[pastelgray]{6}{-2}{$B_1$}
    
    \stripedBox{10,-3}[pastelred]{4}{-2}{$B_2$}
    \stripedBox{0,0}[pastelgreen][pastelorange]{4}{-3}{$P_2$}
    \stripedBox{0,-2}[pastelgreen]{4}{-1}{}
    \stripedBox{10,0}{4}{-3}{}    
    \stripedBox{0,-3}{4}{-2}{}    
\end{scope}
    \begin{scope}[shift={(25,-13)}]
    \stripedBox{4,0}[pastelorange]{10}{-2}{$A$}
    \stripedBox{8,-2}[pastelgray]{6}{-1}{$B_1$}
    \stripedBox{4,-3}[pastelgray]{6}{-2}{$B_1$}
    
    \stripedBox{4,0}[pastelgreen]{4}{-3}{$P_2$}
    \stripedBox{10,-3}[pastelred]{4}{-2}{$B_2$}
    \stripedBox{0,-3}[pastelorange]{4}{-2}{$P_2$}
    \draw[-,line width=3pt] (0,-3) -- (4,-3) -- (4,0) -- (14,0) -- (14,-5) -- (0,-5) -- (0,-3)-- (4,-3);
\end{scope}   
    \begin{scope}[shift={(64,-18)}, rotate=180]
    \stripedBox{4,0}[pastelorange]{10}{-2}{$A$}
    \stripedBox{8,-2}[pastelgray]{6}{-1}{$B_1$}
    \stripedBox{4,-3}[pastelgray]{6}{-2}{$B_1$}
    
    \stripedBox{4,0}[pastelgreen]{4}{-3}{$P_2$}
    \stripedBox{10,-3}[pastelred]{4}{-2}{$B_2$}
    \stripedBox{0,-3}[pastelorange]{4}{-2}{$P_2$}
    \draw[-,line width=3pt] (0,-3) -- (4,-3) -- (4,0) -- (14,0) -- (14,-5) -- (0,-5) -- (0,-3)-- (4,-3);
\end{scope}   

\end{tikzpicture}
    \caption{Procedure for a semithick partition.}
    \label{fig:semithick}
\end{figure}
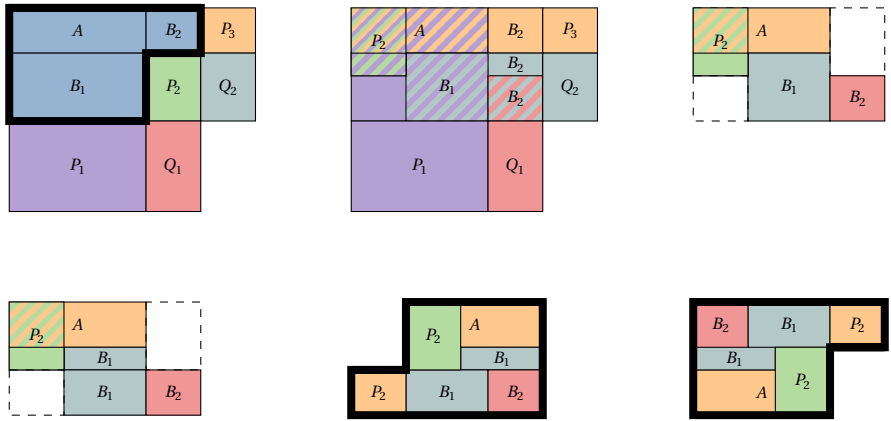

\vspace{1cm}
\section{A new identity between hooks of different partitions}
\label{sec:non:sottili}

\subsection{Hook-to-strip} 
\label{sec:hook:strips}
We prove in this section \Cref{theorem:hooks:strips:intro} from the introduction.

\begin{theorem}\label{theorem:hooks:strips}
 Let  $d \geq \ell > 0$. There exists a bijection of sets
    \begin{equation*}
        S_{d,\ell} \longleftrightarrow S'_{d-\ell,\ell},
    \end{equation*}
    which preserves the positions of both hands and feet of the hooks, in particular preserving content sets and hook types.
    
    \begin{proof}
        We provide a bijective proof, which we pictorially represent in  \cref{fig:hook:strip}  as a guiding example.
       Consider  $(\lambda, h) \in S'_{d-\ell, \ell}$. Notice that the content set of the external hook $h$ consists of the following $|h| = \ell$ consecutive integers
       \[\{k, k+1, k+2, \dots, k+|h|-1\},\] 
       for some integer $k$. 
        Fixing hand and foot of the hook $h$, there is a unique \emph{strip} $s$ external to $\lambda$, of total size  $|h|$, with the extremal boxes of $h$ and such  that
    \[\mu=\lambda\cup s\]
    is again a partition, of size $d$. The strip $s$ is now an internal strip of $\mu$, and   its content set coincides with the one of  $h$. Once again, there exists a unique internal hook $g$ of $\mu$ of hook length $|h|$ and with the same hand and foot as the strip $s$ extremal boxes. Hence $(\mu,g) \in S_{d, \ell}$, and the content set of $g$ is the same of $s$, and hence of $h$. Each step of this correspondence can be clearly reversed, therefore exhibiting the required bijection.
    \end{proof}
\end{theorem}
\begin{example}
 For $\ell=d$   there is a bijection  
    \[
    S_{d,d} \longleftrightarrow S'_{0,d}
    \]
    preserving the content set of the hooks. In other words, external hooks of the empty partition correspond to $L$-shaped partitions, i.e. partitions of the form $\lambda = (\lambda_1, 1, 1, \dots, 1)$ for some positive integer $\lambda_1$, preserving their content set.
\end{example}

\begin{example}
    Consider the situation in   \cref{fig:hook:strip}. The first diagram shows the pair $(\lambda, h)$ for the partition $\lambda = (6,4,3,3,1,1,1)$ with $|\lambda| = 19$ and the external hook $h$ of size  $|h|=7$. By content invariance,  we can \emph{push} $h$ to an external strip $s$ of $\lambda$, so that the content sets of $h$ and $s$ remain invariant. Adding the external strip $s$ to  $\lambda$ defines a new partition $\mu= (6,5,5,4,4,1,1)$ (second diagram), with size $|\mu| = |\lambda| + |h| = 19 + 7 = 26$. Once more, the strip can be pushed to an internal hook, preserving their content set.
    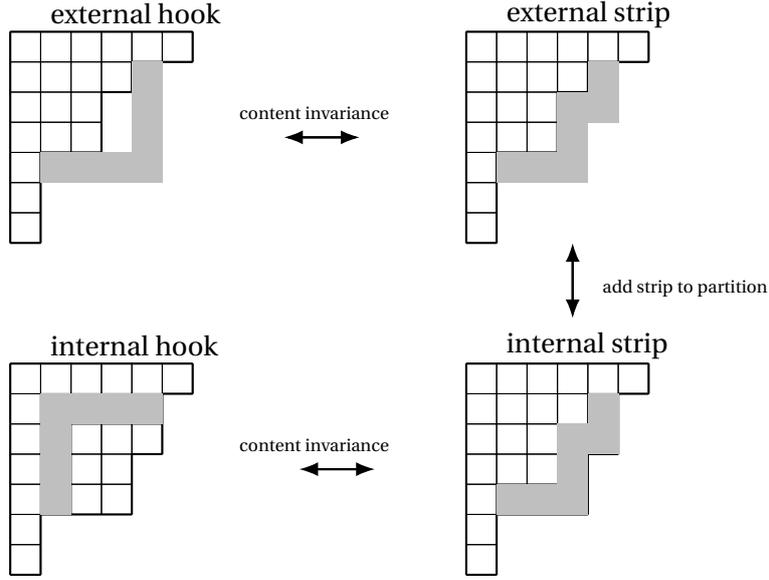
\begin{figure}[h!]
    \centering
\begin{tikzpicture}[scale=0.4]

\def\rowheights{{6,5,5,4,4,1,1}}
\def\rowheightsp{{6,4,3,3,1,1,1}}

\begin{scope}
    \foreach \i [evaluate=\i as \y using 7-\i] in {1,...,7} {
        \pgfmathparse{\rowheightsp[\i-1]}
        \let\rowlen\pgfmathresult
        \draw[thick] (0,\y) -- (0,\y+1);
        \draw[thick] (\rowlen,\y) -- (\rowlen,\y+1);
        \foreach \x in {0,...,6} {
            \ifnum\x<\rowlen
                \draw[thin] (\x,\y) rectangle ++(1,1);
            \fi
        }
    }

    \draw[thick] (0,7) -- (6,7);
    \draw[thick] (0,0) -- (1,0);
    \draw[thick] (1,3) -- (3,3);
    \draw[thick] (3,5) -- (4,5);
    \draw[thick] (4,6) -- (6,6);

    \foreach \coord in {(1,2),(2,2),(3,2),(3,3),(3,4),(4,4),(4,5)} {
        \fill[gray!50] \coord rectangle ++(1,1);
    }

    \node[anchor=west] at (1,7.6) {external strip};
\end{scope}

\draw[-{Latex[length=2.5mm]}, thick] (3.5,-0.5) -- (3.5,-2.5);
\draw[-{Latex[length=2.5mm]}, thick] (3.5,-2.0) -- (3.5,-0.0);
\node at (7.2,-1.5) {\text{\tiny add strip to partition}};

\draw[-{Latex[length=2.5mm]}, thick] (-5.5,3.5) -- (-3.5,3.5);
\draw[-{Latex[length=2.5mm]}, thick] (-4.0,3.5) -- (-6.0,3.5);
\node at (-5,4.3) {\text{\tiny content invariance}};

\begin{scope}[yshift=-11cm]
    \foreach \i [evaluate=\i as \y using 7-\i] in {1,...,7} {
        \pgfmathparse{\rowheights[\i-1]}
        \let\rowlen\pgfmathresult
        \draw[thick] (0,\y) -- (0,\y+1);
        \draw[thick] (\rowlen,\y) -- (\rowlen,\y+1);
        \foreach \x in {0,...,6} {
            \ifnum\x<\rowlen
                \draw[thin] (\x,\y) rectangle ++(1,1);
            \fi
        }
    }

    \draw[thick] (0,7) -- (6,7);
    \draw[thick] (0,0) -- (1,0);
    \draw[thick] (1,2) -- (4,2);
    \draw[thick] (4,4) -- (5,4);
    \draw[thick] (5,6) -- (6,6);

    \foreach \coord in {(1,2),(2,2),(3,2),(3,3),(3,4),(4,4),(4,5)} {
        \fill[gray!50] \coord rectangle ++(1,1);
    }

    \node[anchor=west] at (1,7.6) { internal strip };


\draw[-{Latex[length=2.5mm]}, thick] (-3.5,3.5) -- (-5.5,3.5);
\draw[-{Latex[length=2.5mm]}, thick] (-5,3.5) -- (-3,3.5);
\node at (-5,4.3) {\text{\tiny content invariance}};
    
\end{scope}

\begin{scope}[xshift=-15cm]
    \foreach \i [evaluate=\i as \y using 7-\i] in {1,...,7} {
        \pgfmathparse{\rowheightsp[\i-1]}
        \let\rowlen\pgfmathresult
        \draw[thick] (0,\y) -- (0,\y+1);
        \draw[thick] (\rowlen,\y) -- (\rowlen,\y+1);
        \foreach \x in {0,...,6} {
            \ifnum\x<\rowlen
                \draw[thin] (\x,\y) rectangle ++(1,1);
            \fi
        }
    }

    \draw[thick] (0,7) -- (6,7);
    \draw[thick] (0,0) -- (1,0);
    \draw[thick] (1,3) -- (3,3);
    \draw[thick] (3,5) -- (4,5);
    \draw[thick] (4,6) -- (6,6);

    \foreach \coord in {(1,2),(2,2),(3,2),(4,2),(4,3),(4,4),(4,5)} {
        \fill[gray!50] \coord rectangle ++(1,1);
    }
    \node[anchor=west] at (1,7.6) {external hook};
\end{scope}

\begin{scope}[xshift=-15cm, yshift=-11cm]
    \foreach \i [evaluate=\i as \y using 7-\i] in {1,...,7} {
        \pgfmathparse{\rowheights[\i-1]}
        \let\rowlen\pgfmathresult
        \draw[thick] (0,\y) -- (0,\y+1);
        \draw[thick] (\rowlen,\y) -- (\rowlen,\y+1);
        \foreach \x in {0,...,6} {
            \ifnum\x<\rowlen
                \draw[thin] (\x,\y) rectangle ++(1,1);
            \fi
        }
    }

    \draw[thick] (0,7) -- (6,7);
    \draw[thick] (0,0) -- (1,0);
    \draw[thick] (1,2) -- (4,2);
    \draw[thick] (4,4) -- (5,4);
    \draw[thick] (5,6) -- (6,6);

    \foreach \coord in {(1,2),(1,3),(1,4),(1,5),(2,5),(3,5),(4,5)} {
        \fill[gray!50] \coord rectangle ++(1,1);
    }

    \node[anchor=west] at (1,7.6) {internal hook};
\end{scope}

\end{tikzpicture}
\caption{The hook-to-strip correspondence.}
\label{fig:hook:strip}
\end{figure}

\end{example}
As an application of \Cref{theorem:hooks:strips}, we derive a new correspondence between the counting problem of reverse and skew plane partitions. Similarly to  \eqref{eqn:multvar}, given an internal or external  hook $h$ of $\lambda$, we set    \begin{align*}
\mathbf{p}_{h}=\prod_{(i,j)\in h}q_{j-i}.       
   \end{align*}

\begin{corollary} \label{cor:hook:strip}
Let $d, \ell\geq 0$. We have an identity
     \begin{equation*}
         \prod_{(\lambda, h) \in S_{d,\ell}}\frac{1}{(1 - \mathbf{p}_{h})}
         =
         \prod_{(\lambda', h') \in S'_{d-\ell,\ell}}\frac{1}{(1 - \mathbf{p}_{h'})}.
     \end{equation*}
\end{corollary}
     \begin{proof}
           Let $S, S'$ be collections of subsets of boxes $\Box\in \BZ^2_{\geq 0}$. For $T\in S, S'$, set the multivariable
    \[
    {\bf q}_T =\prod_{\Box \in T} q_{\Box}.
    \]
    Then, it easily follows that
    \begin{equation*}
        \sum_{T \in S}{\bf q}_T = \sum_{T' \in S'}{\bf q}_{T'} \iff S = S'.
    \end{equation*}
    and by taking the plethystic exponential we obtain
    \begin{equation*}
        \prod_{T \in S}\frac{1}{(1 - {\bf q}_T)} = \prod_{T' \in S'}\frac{1}{(1 - {\bf q}_{T'})} \iff S = S'.
    \end{equation*}
    Imposing the variable specialisation $q_{\Box} = q_{c(\Box) = j-i}$, the proof of the corollary follows from \Cref{theorem:hooks:strips}.
\end{proof}

Exploiting the fact that \Cref{theorem:hooks:strips} preservs hook types, we immediately get the following corollary.

\begin{corollary} \label{cor:hook:strip:refined}
Let $d \geq \ell > 0$. We have an identity
     \begin{equation*}
         \prod_{(\lambda, \Box) \in S_{d,\ell}}\frac{1}{(1 - x^{a_{\lambda}(\Box) + 1} y^{\ell_{\lambda}(\Box)})}
         =
         \prod_{(\lambda', \Box') \in S'_{d-\ell,\ell}}\frac{1}{(1 - x^{a_{\lambda}(\Box') + 1} y^{\ell_{\lambda}(\Box')})}.
     \end{equation*}
\end{corollary}

As an application, we can combine the result above and \cref{theorem:hooks:strips} to get the following.

\begin{prop}\label{prop: ultimate wall refined} We have the identity
    \begin{equation*}
        \prod_{\substack{(\lambda, \Box) \in S_{d+\ell,\ell} \\ \ell >0 }}\frac{1}{1 - x^{a_{\lambda}(\Box) + 1} y^{\ell_{\lambda}(\Box)}}
        =
        \left(
        \prod_{\Box'' \in \mathcal{H}'(\varnothing)} \frac{1}{1 - x^{a_{\varnothing}(\Box'') + 1} y^{\ell_{\varnothing}(\Box'')}}
        \right)^{|\mathcal{P}_{d}|}
        \!\!\!\!\!\!
        \prod_{\substack{(\lambda, \Box') \in S_{d,\ell} \\ \ell >0 }}\frac{1}{1 - x^{a_{\lambda}(\Box') + 1} y^{\ell_{\lambda}(\Box')}}.
    \end{equation*}
    \begin{proof}
    Taking the product over all positive $\ell$ of the statement of \cref{cor:hook:strip:refined}, we compute
    \begin{align*}
        \prod_{\substack{(\lambda, \Box) \in S_{d+\ell,\ell} \\ \ell >0 }} & \frac{1}{1 - x^{a_{\lambda}(\Box) + 1} y^{\ell_{\lambda}(\Box)}}
        =
        \prod_{\substack{(\lambda, \Box') \in S'_{d,\ell} \\ \ell > 0}}\frac{1}{1 - x^{a_{\lambda}(\Box') + 1} y^{\ell_{\lambda}(\Box')}}
        \\
        &= 
        \prod_{\lambda \in \mathcal{P}_{d}} 
        \left(
        \prod_{\Box' \in \mathcal{H}'(\lambda)} 
        \frac{1}{1 - x^{a_{\lambda}(\Box') + 1} y^{\ell_{\lambda}(\Box')}} 
        \right)
        \\
        &= 
        \prod_{\lambda \in \mathcal{P}_{d}}
        \left(\prod_{\Box'' \in \mathcal{H}'(\varnothing)} \frac{1}{1 - x^{a_{\lambda}(\Box'') + 1} y^{\ell_{\lambda}(\Box'')}}
        \prod_{\Box' \in \mathcal{H}(\lambda)} \frac{1}{1 - x^{a_{\lambda}(\Box') + 1} y^{\ell_{\lambda}(\Box')}} \right) 
        \\
        &= 
        \left(
        \prod_{\Box'' \in \mathcal{H}'(\varnothing)} \frac{1}{1 - x^{a_{\varnothing}(\Box'') + 1} y^{\ell_{\varnothing}(\Box'')}}
        \right)^{|\mathcal{P}_{d}|}
        \left( \prod_{\lambda \in \mathcal{P}_{d}} \prod_{\Box' \in \mathcal{H}(\lambda)} \frac{1}{1 - x^{a_{\lambda}(\Box') + 1} y^{\ell_{\lambda}(\Box')}} \right) 
        \\
        &= 
        \left(
        \prod_{\Box'' \in \mathcal{H}'(\varnothing)} \frac{1}{1 - x^{a_{\varnothing}(\Box'') + 1} y^{\ell_{\varnothing}(\Box'')}}
        \right)^{|\mathcal{P}_{d}|}
        \!\!\!\!\!\!
        \prod_{\substack{(\lambda, \Box) \in S_{d,\ell} \\ \ell >0 }}\frac{1}{1 - x^{a_{\lambda}(\Box) + 1} y^{\ell_{\lambda}(\Box)}},
    \end{align*}
    where in the third line we applied \Cref{thm: Bessen_original}.
    This concludes the proof of the proposition.    
    \end{proof}
\end{prop}

\vspace{1cm}
\section{Fock space interpretation of \cref{theorem:hooks:strips}}
\label{sec:Fock}
We  briefly recall the definition of the fermionic and bosonic Fock spaces and their most used operators. We refer to ~\cite{okounkov2006gromov} for more details.

Let $V$ be an infinite-dimensional  $\mathbb{C}$-vector space with basis $\left\{\underline{s} \mid s \in \mathbb{Z} + \frac{1}{2} \right\}$. The \emph{semi-infinite wedge space}, denoted by $\mathcal{V} = \Lambda ^ \frac{\infty}{2} V$, has a basis defined by
\[ 
v_S := \{\underline{s_1} \wedge \underline{s_2} \wedge \underline{s_3} \wedge \cdots \mid s_1 > s_2 > s_3 > \cdots \}, 
\]
where $S = \{s_1 > s_2 > \cdots \} \subset \mathbb{Z} + \frac{1}{2}$ is such that the sets
\[
S_+ = S \setminus \left(\mathbb{Z}_{\le 0} - \tfrac{1}{2} \right) \quad \text{and} \quad S_- = \left( \mathbb{Z}_{\le 0} - \tfrac{1}{2}\right) \setminus S
\]
are finite. We equip $\mathcal{V}$ with the inner product $\left( \cdot, \cdot\right)$ defined by the elements $v_{S}$ being orthonormal.

There exists a unique $c \in \mathbb{Z}$ such that $s_k + k - 1/2 = c$ for $k$ sufficiently large; this constant $c$ is called the \emph{charge}. The charge-zero subspace, denoted $\mathcal{V}_0 \subset \mathcal{V}$, is spanned by semi-infinite wedge products of the form 
\[
\underline{\lambda_1 - \tfrac{1}{2}} \wedge \underline{\lambda_2 - \tfrac{3}{2}} \wedge \underline{\lambda_3 - \tfrac{5}{2}} \wedge \cdots,
\]
indexed by partitions $\lambda \in \mathcal{P}$. The basis element in $\mathcal{V}_0$ corresponding to the empty partition,
\[
v_{\varnothing} = \underline{-\tfrac{1}{2}} \wedge \underline{-\tfrac{3}{2}} \wedge \underline{-\tfrac{5}{2}} \wedge \cdots,
\]
is called the \emph{vacuum vector} and plays a special role. Similarly, the dual of the vacuum vector with respect to the inner product $\left( \cdot, \cdot \right)$ is called the \emph{covacuum vector}. 

We also define the following operators that will be used in the rest of the paper.
\begin{definition}
For $k \in \mathbb{Z} + \tfrac{1}{2}$, the \emph{fermionic operator} $\psi_k$ is defined by
\[
\psi_k v_S = \underline{k} \wedge v_{S}.
\]
The operator $\psi_k^*$ is defined to be the adjoint of $\psi_k$ with respect to the inner product. The \emph{normally ordered product} is defined by
\[
: \psi_i \psi_j^*: \ = \begin{cases} 
	\psi_i \psi_j^*, & \text{if } j > 0, \\
	-\psi_j^* \psi_i, & \text{if } j < 0. 
	\end{cases}
\]
\end{definition}

\begin{definition} 
For a non-negative integer $r$, define the operator
\[
\mathcal{F}_r := \sum_{k \in \mathbb{Z} + \frac{1}{2}} \frac{k^r}{r!} :\psi_k \psi_k^* :
\]
The operator $\mathcal{F}_0$ is called the \emph{charge operator}. We say that an operator $\mathcal{O}$ acting on $\mathcal{V}$ has charge $c \in \mathbb{Z}$ if 
\[
[\mathcal{O}, \mathcal{F}_0] = c \, \mathcal{O}. 
\]
The operators $\mathcal{F}_r, :\psi_i \psi_j^* :$ have charge zero, while the operators $\psi_i$ and $\psi_j^*$ have charge equal to $1$ and $-1$, respectively.
The operator $\mathcal{F}_1$ is called the \emph{energy operator}. We say that an operator $\mathcal{O}$ acting on $\mathcal{V}$ has \emph{energy} $E \in \mathbb{Z}$ if 
\[
[\mathcal{O}, \mathcal{F}_1] = E \, \mathcal{O}. 
\]
The operators $:\psi_i \psi_j^* :$ have energy $j - i$. Operators with positive energy annihilate the vacuum, while operators with negative energy are annihilated by the covacuum.

\end{definition}

\begin{definition}
    The \emph{free boson operators} are defined as 
\begin{equation*}
\alpha_E \coloneqq \sum_{k \in \mathbb{Z} + \frac{1}{2}} \, :\psi_{k-E} \psi_k^*:
\end{equation*}
These operators have energy $E$, zero charge, and satisfy the bosonic commutation relation 
\begin{equation*}
	[\alpha_E, \alpha_F] = E \delta_{E + F}.
\end{equation*}

We now enrich slightly the free boson operators to form the generating series we are interested in. The $\bf{q}$-twisted \emph{free boson operators} are defined as
\begin{equation*}
\alpha_E^{\bf{q}} \coloneqq \sum_{k \in \mathbb{Z} + \frac{1}{2}} \, \left(\prod_{j = k+1/2}^{k-E-1/2} q_j \right) \vdots \psi_{k-E} \psi_k^* \vdots
\end{equation*}
where the symbol $\vdots \cdots \vdots$ stands for the usual normal ordering $: \cdots :$ but additional applies the absolute value to the coefficient of the application to each individual Maya diagram.
\end{definition}

Now we are armed to give a completely equivalent statement of \cref{theorem:hooks:strips} in terms of Fock space operators.

\begin{prop}
    For $\ell > 0$ we have
\begin{equation*}
    \sum_{\substack{\lambda \in \mathcal{P}_d \\ \mu \in \mathcal{P}_{d-\ell}}} 
    \Big \langle \mu |
    \alpha_\ell^{\bf{q}} 
    | \lambda \Big \rangle
    = 
    \sum_{\substack{\lambda \in \mathcal{P}_d \\ \mu \in \mathcal{P}_{d-\ell}}} 
    \Big \langle \lambda |
    \alpha_{-\ell}^{\bf{q}} 
    | \mu \Big \rangle.
\end{equation*}
\begin{proof}
    It follows from the action of the free bosons via the Murnagham-Nakayama rule as removing strips of size equal to the energy in all possible ways. The rest follows by the proof of \cref{theorem:hooks:strips}. This concludes the proof of the proposition.
\end{proof}
\end{prop}

\bibliographystyle{amsplain}
\bibliography{The_Bible}

\bigskip
\noindent
{\small{Davide Accadia \\
\address{Università di Trieste, Dipartimento MIGe, Via Valerio 12/1, 34127, Trieste, Italy
\& Istituto Nazionale di Fisica Nucleare (INFN), Sezione di Trieste, Italy.} 
\\
\href{mailto:davide.accadia@phd.units.it}{\texttt{davide.accadia@phd.units.it}}
}}

\bigskip
\noindent
{\small{Danilo Lewa\'nski \\
\address{Università di Trieste, Dipartimento MIGe, Via Valerio 12/1, 34127, Trieste, Italy \& Istituto Nazionale di Fisica Nucleare (INFN), Sezione di Trieste, Italy.} 
\\
\href{mailto:danilo.lewanski@units.it}{\texttt{danilo.lewanski@units.it}}
}}

\bigskip
\noindent
{\small{Sergej Monavari \\
\address{Dipartimento di Matematica “Tullio Levi-Civita”, Università degli Studi di Padova, Via Trieste 63, 35121 Padova, Italy} \\
\href{mailto:sergej.monavari@math.unipd.it}{\texttt{sergej.monavari@math.unipd.it}}
}}

 \end{document}